\documentclass[11pt]{article}
\usepackage{amsmath,amsfonts,amsthm,amssymb,color}

\usepackage{fullpage}
\usepackage{subfiles}
\usepackage{tikz}
\usepackage{standalone}
\usepackage{graphicx}
\usepackage{wrapfig}
\usepackage{hyperref}
\newcommand{\maxroot}[1]{\mathrm{maxroot} \left\{ #1\right\}}

\newtheorem{thm}{Theorem}[section]
\newtheorem{cor}[thm]{Corollary}
\newtheorem{lemma}[thm]{Lemma}

\newtheorem{prop}[thm]{Proposition}

\theoremstyle{definition}
\newtheorem{defn}{Definition}
\newtheorem{rmk}{Remark}

\newcommand{\Strans}[1]{\mathbb{S}#1}
\newcommand{\recsum}{\boxplus}
\newcommand{\mysum}[2]{\boxplus_{#1}^{#2}}

\newcommand{\mm} {m}
\newcommand{\nn}{n}
\newcommand{\dd}{d}
\newcommand{\kk}{k}
\newcommand{\AND}{\quad\text{and}\quad}

\newcommand{\R}{\mathbb{R}}
\newcommand{\perms}{\mathcal{P}}
\newcommand{\orth}{\mathcal{O}}
\newcommand{\reps}{\mathcal{A}}

\newcommand{\rframe}[2]{V_{#1}(\mathbb{R}^{#2})}
\newcommand{\test}[1]{\delta_{\left\{ #1 \right\}}}

\newcommand{\one}{\vec{1}}

\newcommand{\charpoly}[1]{\chi \left[ #1 \right]}
\newcommand{\mydet}[1]{\det \left[ #1 \right]}

\newcommand{\twobytwo}[4]{\begin{pmatrix} #1 & #2 \\ #3 & #4 \end{pmatrix}}

\newcommand{\offdiag}[1]{\twobytwo{0}{{#1}}{{#1}^T}{0}}
\newcommand{\diag}[2]{\twobytwo{{#1}}{0}{0}{{#2}}}
\newcommand{\topleft}[1]{\diag{{#1}}{0}}

\newcommand{\expect}[2]{\mathbb{E}_{#1} \left\{ #2 \right\}}
\newcommand{\expectt}[3]{\expect{\substack{#1 \\ #2}}{#3}}

\newcommand{\minor}[3]{\left[ #1 \right]_{#2, #3}}

\newcommand{\arxiv}[1]{\href{https://arxiv.org/abs/#1}{arXiv:#1}}

\title{Existence and polynomial time construction of biregular, 
bipartite Ramanujan graphs of all degrees}
\date{\today}
\author{Aurelien Gribinski\\ 
Princeton University
\and Adam W. Marcus\thanks{Research done under the support of NSF CAREER grant DMS-1552520 and a Von Neumman Fellowship at the Institute of Advanced Study, NSF grant DMS-1128155.}\\
EPFL
}

\begin{document}
\maketitle

\begin{abstract}
We prove that there exist bipartite, biregular Ramanujan graphs of every degree 
and every number of vertices provided that the cardinalities of the two sets of 
the bipartition divide each other.
This generalizes the main result of Marcus, Spielman, and Srivastava 
\cite{MSS4} and, similar to theirs, the proof is 
based on the analysis of expected polynomials.
The primary difference is the use of some new 
machinery involving rectangular convolutions, developed in a companion paper 
\cite{GM1}. 
We also prove the constructibility of such graphs in polynomial 
time in the number of vertices, extending a result of Cohen \cite{MCohen} to 
this 
biregular case.
\end{abstract}

\section{Introduction}

Ramanujan graphs are expander graphs with optimal asymptotic spectral 
properties.
This paper proves two new results concerning Ramanujan graphs. 
Our first result extends the methods of \cite{MSS4} to show the existence of 
{\em unbalanced} bipartite Ramanujan graphs of all sizes and all degrees (under 
the assumption that the sizes 
of the two sets of the bipartition divide each other).
One notable aspect of the result in \cite{MSS4} was the result of \cite{MCohen} 
that showed how one could turn the existence proof in \cite{MSS4} into a 
polynomial construction.
Our second result is an extension of this construction to the biregular 
graphs in the first result.
\subsection{Motivation}

$(c, d)$-biregular bipartite graphs generalize $d$-regular bipartite graphs in 
an obvious way 
--- they allow for graphs which have partitions of different sizes.
This makes them better suited for applications such as deep neural networks and 
LDPC codes.
Similar to the $d$-regular case, however, specific graphs can be better or 
worse suited for a particular applications.
In many situations, the ``suitability'' can be expressed as a function of 
the expansion properties of the graph.
While most expansion properties are NP-complete to compute, there is a 
well-know relation between many of these expansions properties and the 
so-called {\em spectral gap} via Cheeger's inequality.

To be specific, note that the adjacency matrix of a $(c, d)$-biregular 
bipartite graph has the form
\[
B= \offdiag{A}
\]
where $A$ is an $\mm \times \nn$ rectangular matrix with $d$ $1$'s in each 
row and $c$ $1$'s in each column.
Such matrices necessarily have eigenvalues $\pm \sqrt{cd}$ (often called the 
``trivial'' eigenvalues); the spectral gap is then the distance from these 
trivial eigenvalues to the next largest eigenvalue (in absolute value).
Since the eigenvalues of $B$ are symmetric about $0$, maximizing the spectral 
gap can be simplified to minimizing the size of the second eigenvalue.

There is a well-known lower bound on the size of the second eigenvalue of 
biregular, bipartite graphs (at least asymptotically) that is an extension of 
the more well-known Alon--Boppana bound \cite{fengli}:

\begin{thm}[Feng--Li]\label{fengli}
Let $G_1, G_2, \dots$ be a sequence of $(c, d)$ biregular, bipartite graphs and 
suppose there exists a real number $\theta$ for which the second eigenvalue of 
$G_i$ is at most $\theta$ for all $i$.
Then 
\[
\theta \geq \sqrt{c-1} + \sqrt{d-1}.
\]
\end{thm}

Following \cite{LPS}, we call a sequence of graphs for which $G_i \leq \theta$  
{\em Ramanujan}.
The goal, then, is to find such a sequence.
\begin{rmk}
Note that in the special case $c=d$, a $(c, d)$-biregular graph is simply a 
$d$-regular bipartite graph, and so \eqref{eq:raman} reduces to the well-known 
$2 \sqrt{d-1}$ appearing in the Alon--Boppana bound.
\end{rmk}

\subsection{Previous work}

Ramanujan graphs first appeared in number theory in the work of \cite{Mar} and 
\cite{LPS}, who proved the existence of families of Ramanujan graphs whose 
degree $d = p+1$ where $p$ is prime, and this was later extended in \cite{Mor} 
to the case where $p$ is a prime power. 
More recently, it was proven that for any integer $d$, there are Ramanujan 
bipartite graphs whose degree is $d$ in \cite{MSS1}, but their  specific sizes 
and without construction.
This was extended in \cite{HPS} to a wider collection of sizes but still 
without construction.
Finally, \cite{MSS4} proved the existence of bipartite Ramanujan graphs of any 
degree $d$ and any size $2n$, and this was later turned into a polynomial 
construction in \cite{MCohen}.

For certain applications, however, unbalanced bipartite graphs are often more 
suitable for applications in computer science than balanced ones. 
For example, the recent trend of ``deep'' neural networks favors an 
architecture with multiple layers of varying sizes. 
The majority of the results regarding $d$-regular bipartite graph have been 
extended to the unbalanced case.
The methods of \cite{LPS} and \cite{Mor} were adapted by Ballantine et al. 
\cite{Bal} to give an explicit construction for $(p + 1, p^3 +1)$-biregular 
bipartite graphs when $q$ is a power of a prime number (so in a number and 
group theory perspective). 
The methods of \cite{MSS1} and \cite{HPS} can be used directly to show the 
existence of $(c, d)$-biregular bipartite graphs for arbitrary $c$ and $d$ but 
(again) without construction.
The main contribution of this paper is to complete the picture by extending the 
results of $\cite{MSS4}$ and $\cite{MCohen}$ to the unbalanced case in the case 
when $d$ is an integer multiple of $c$ (or vice versa).

While Ramanujan graphs are, in some sense, the ``optimal'' expanders one could 
hope to provably build, it should be noted that in practice having an {\em 
almost-Ramanujan} graph is often sufficient.
In this respect, life is much easier --- for the $d$-regular bipartite case, 
\cite{fried} showed that, for any fixed $\epsilon$, the probability that a 
random $d$-regular bipartite graph on $n$ vertices has nontrivial eigenvalues 
in the 
interval
\[
[ -2 \sqrt{d-1} - \epsilon, 2 \sqrt{d-1} + \epsilon ]
\]
goes to $1$ as $n \to \infty$.
This was extended to $(c, d)$-biregular bipartite graphs in \cite{Dum}.
However, we do note that although one might expect almost-Ramanujan and 
Ramanujan graphs to have almost-identical properties, there are certain proof 
techniques (related to the Ihara Zeta function) for which Ramanujan graphs are 
distinctly better, for reasons that are beyond the scope of this article (see 
\cite{Bal2}).

\subsection{Summary of the results and outline of the proof}

We will call a biregular, bipartite graph $G = (V, E)$ an $(n, k, d$)-graph if 
the vertices in $G$ can be partitioned into two sets $V = I_1 \cup I_2$ where 
\begin{itemize}
\item $|I_1| = k n$ and $|I_2| = n$
\item $\deg(v) = d$ for all $v \in I_1$ and $\deg(v) = k d$ for all $v \in I_2$
\item all edges in $E$ have one endpoint in $I_1$ and one endpoint in $I_2$
\end{itemize} 
For fixed $k$ and $d$, Theorem~\ref{fengli} implies that the smallest upper 
bound one can hope to get for the second eigenvalue of a sequence of $(n, k, 
d)$-graphs is 
\begin{equation}\label{eq:raman}
\sqrt{d-1} + \sqrt{kd-1}.
\end{equation}
Using the typical convention that, for a Hermitian matrix $B$, $\lambda_k(B)$ 
denotes the $k$th largest eigenvalue (see Section~\ref{sec:notations} for a 
list of notations), our main theorem is:
\begin{thm}\label{thm:main}
For all $\nn, \kk, \dd$, there exists an $(\nn, \kk, \dd)$-graph whose 
adjacency matrix $B$ satisfies 
\[
\lambda_2(B) \leq \sqrt{\dd-1} + \sqrt{\kk\dd-1}.
\]
Furthermore, for fixed $k$ and $d$, we can construct such a graph in time that 
is polynomial in $n$.
\end{thm}

As mentioned previously, the approach builds on the work of \cite{MSS4} and 
\cite{MCohen}.
We first prove existence of such graphs in a manner similar to \cite{MSS4}.
Whereas \cite{MSS4} ``built'' Ramanujan graphs as a union of perfect matchings, 
our graphs will be ``built'' as a union of ``$\kk$-claw  matchings'' (see 
Figure~\ref{fig:k-claw}). 

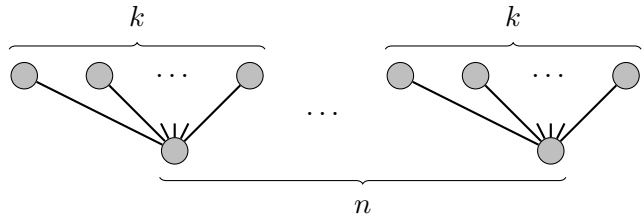
\begin{figure}[h]

\usetikzlibrary{positioning,chains,fit,shapes,calc}
\usetikzlibrary{decorations.pathreplacing}
\tikzstyle{vertex}=[draw,circle,fill=black!25,minimum size=10pt,inner sep=1pt]
\tikzstyle{edge}= [draw,thick,-]

\begin{center}
\begin{tikzpicture}[scale=1]

\def\x{0}
\def\y{0}
\def\val{1}
 
\node[vertex] (v\val) at (\x, \y) {};
\node[vertex] (u1\val) at (\x-2, \y+1) {};
\node[vertex] (u2\val) at (\x-1, \y+1) {};
\node  at (\x, \y+1) {$\dots$};
\node[vertex] (u3\val) at (\x+1, \y+1) {};

\node (fake1\val) at (\x - 0.25, \y + 0.5) {};
\node (fake2\val) at (\x, \y + 0.5) {};
\node (fake3\val) at (\x + 0.25, \y + 0.5) {};

\path[edge] (v\val) edge (u1\val);
\path[edge] (v\val) edge (u2\val);
\path[edge] (v\val) edge (u3\val);

\path[edge] (v\val) edge (fake1\val);
\path[edge] (v\val) edge (fake2\val);
\path[edge] (v\val) edge (fake3\val);

\draw [decorate, decoration={brace, raise=10pt}] (\x - 2.2, \y + 1) -- 
node[above=15pt] {$\kk$} (\x+1.2, \y + 1);

\def\x{5}
\def\y{0}
\def\val{2}
 
\node[vertex] (v\val) at (\x, \y) {};
\node[vertex] (u1\val) at (\x-2, \y+1) {};
\node[vertex] (u2\val) at (\x-1, \y+1) {};
\node at (\x, \y+1) {$\dots$};
\node[vertex] (u3\val) at (\x+1, \y+1) {};

\node (fake1\val) at (\x - 0.25, \y + 0.5) {};
\node (fake2\val) at (\x, \y + 0.5) {};
\node (fake3\val) at (\x + 0.25, \y + 0.5) {};

\path[edge] (v\val) edge (u1\val);
\path[edge] (v\val) edge (u2\val);
\path[edge] (v\val) edge (u3\val);

\path[edge] (v\val) edge (fake1\val);
\path[edge] (v\val) edge (fake2\val);
\path[edge] (v\val) edge (fake3\val);

\draw [decorate, decoration={brace, raise=10pt}] (\x - 2.2, \y + 1) -- 
node[above=15pt] {$\kk$} (\x+1.2, \y + 1);

\node at (2, 0.5) {$\dots$};

\draw [decorate, decoration={brace, raise=10pt, mirror}] (-0.2, 0) -- 
node[below=15pt] {$n$} (5.2, 0);

\end{tikzpicture}
\end{center}

\caption{A k-claw matching on $n \times kn$ vertices}\label{fig:k-claw}
\end{figure}
For convenience,  denote a ``canonical'' $\kk$-claw matching: the one 
with adjacency matrix
\[
M = \offdiag{I_{\nn}^{[\kk]}}
\quad \text{where} \quad 
I_{\nn}^{[\kk]} = 
\begin{pmatrix} I_{\nn}\\I_{\nn}\\\vdots\\I_{\nn} \end{pmatrix} \bigg\} 
\text{~$\kk$ copies}
\]
An arbitrary $\kk$-claw matching can be obtained from the canonical one by 
permuting the vertices.
This corresponds to multiplying $I_n^{[k]}$ on the left by some permutation 
matrix $P \in 
\perms_{kn}$ and on the right by some $S \in \perms_n$ (where $\perms_k$ 
denotes 
the collection of $k \times k$ permutation matrices).
Hence the union of $d$ $\kk$-claw matchings has an adjacency matrix of the form
\[
B 
= \sum_{i=1}^{\dd} (P_i \oplus S_i) M  (P_i \oplus S_i )^T
= 
\sum_{i=1}^{\dd} \offdiag{(P_i I_{\nn}^{[\kk]}S_i)}.
\]
Our approach will be to draw $P_i$ and $S_i$ uniformly and independently from 
$\perms_{kn}$ and $\perms_n$ (respectively), and to show that
\[
 \mathbb{P} \left( \lambda_2(B)  \leq \sqrt{\dd-1} + \sqrt{\kk\dd-1}  \right) > 
 0
\]
The proof will use the ``method of interlacing polynomials'' first introduced 
in \cite{MSS1} to link this event to the expected characteristic polynomial.
In particular, we prove the following proposition in Section~\ref{convo}:
\begin{prop} [Expected eigenvalue bound]  \label{Ebound}
The polynomial
\begin{equation}\label{eq:exp_polynomial}
\expectt{P_1, \dots,P_{\dd}}{S_1,\dots,S_{\dd}}{\charpoly{B}}
\end{equation}
is real rooted and
\[
  \lambda_2(B) \leq \mu_2 \left( \expectt{P_1, 
  \dots,P_{\dd}}{S_1,\dots,S_{\dd}}{\charpoly{B}} \right)
\]
with nonzero probability (where $\mu_2$ denotes the ``second largest root'' 
function).
\end{prop}
It is therefore sufficient to prove that
\[
\mu_2\left( \expectt{P_1, \dots,P_{\dd}}{S_1,\dots,S_{\dd}}{\charpoly{B}} 
\right) \leq \sqrt{\dd-1} + \sqrt{\kk\dd-1}
\]
which we show in Section~\ref{bound} using a ``rectangular'' polynomial 
convolution introduced in \cite{GM1}.

We then show how the method used in \cite{MCohen} to make \cite{MSS4} 
constructive can be extended to our situation, resulting in a polynomial 
construction.
Specifically, 
\begin{thm}
For fixed $\kk$ and $\dd$, there exists an algorithm that constructs a 
Ramanujan $(\nn,\kk,\dd)$-graph in time $O(\nn^C)$ where $C = C(k, d)$ is an 
explicit function of $\kk$ and $\dd$.
\end{thm}
\begin{rmk}
We stress that this really is a ``polynomial time algorithm'' in name only --- 
the constant 
$C$ can be (and actually is) quite large.
So in practice, this can only be implemented in the case that both $\kk$ and 
$\dd$ are (quite) small.
\end{rmk}

The general idea of the algorithm is to use the interlacing family tree 
constructed in the existence proof. 
If we consider the root node to correspond to the empty graph, then the 
children of any node in the tree correspond to the graph where a single edge 
has been added.
Every $k$ levels in the tree consists of adding a new $\kk$-claw and every 
$\nn\kk$ levels corresponds to adding a new $\kk$-claw matching.
The leaves (at depth $\dd\kk\nn$) will correspond to all biregular, bipartite 
graphs that have adjacency matrix
\[
\sum_{i=1}^{\dd} (P_i \oplus S_i) M  (P_i \oplus S_i )^T.
\]

Now we can assign polynomials to each node --- the leaves will have the 
characteristic polynomial associated to biregular, bipartite graph assigned to 
it: 
\[
\charpoly{\sum_{i=1}^{\dd} (P_i \oplus S_i) M  (P_i \oplus S_i )^T}
\]
and the non-leaf nodes will be an average of all of the leaf 
nodes descended from it (in particular, the polynomial at the root will get 
the polynomial \eqref{eq:exp_polynomial} above).

The fact that this forms an interlacing family implies two things:
\begin{itemize}
\item Every polynomial associated to a node has only real roots
\item For every non-leaf node polynomial $p$, one of its children has a 
polynomial $q$ with $\mu_2(q) \leq \mu_2(p)$ ($\mu_2$ denoting the second 
largest root).
\end{itemize}
The algorithm goes down the tree $k$ levels at a time.
At each step, it computes all of the characteristic polynomials associated to 
the children $k$ generations down, picks the one that has the best root, and 
continues in this manner.
After $\nn \dd$ steps, we will return the leaf node that is found.
There are two things that need to be checked:
\begin{itemize}
\item At each level the number of characteristic polynomials we need to check 
is polynomial in $n$
\item We can compute these characteristic polynomials efficiently.
\end{itemize}
The first item is straightforward, and the second is straightforward when we 
are at levels that are a multiple of $\kk \nn$ (so unions of complete $k$-claw 
matchings).
At these levels, the expected characteristic polynomials are summed over entire 
copies of the symmetric group, and so this can be done easily using convolution 
formulas.
So it remains to show that this can be done at intermediate steps, and this 
turns out to be far more nontrivial (and is the major contribution of 
\cite{MCohen}).

Assume we are at level for which $\ell$ different $\kk$-claws still need to be 
found to complete the current matching. 
Let $A$ be the adjacency matrix at the most recent level that was a multiple of 
$\kk \nn$ (possible the zero matrix).
In order to take the next step, we will see that it suffices to compute the 
following polynomials in efficiently:
\[
\expectt
{P \in \perms_{\kk l}}
{S \in \perms_{l} }
{
\charpoly{
    \offdiag{A}
    +
    \diag{P \oplus I}{S \oplus I}
    \offdiag{C_{\ell} }
    \diag{P \oplus I}{S \oplus I}^T
}}
\] 
where $C_\ell$ consists of $k$ copies of the identity matrices $I_\ell$ stacked 
up (and the rest $0$): 
The strategy is to bring ourselves back to averaging over entire groups so that 
we can use symmetry and cancellations. 
The idea in \cite{MCohen} is that this can be accomplished in a 
generating-function manner using a cleverly chosen trivariate 
characteristic polynomial:
\[
\mydet{ x I + \diag{y I_s}{I_{\nn\kk-s}} B \diag{z I_r}{I_{\nn-r}} }.
\]
The $y$ and $z$ variables effectively distinguish which parts of $B$ contribute 
to certain coefficients in the characteristic polynomial, so this allows us to 
replace $C_\ell$ with the much nicer $M$ and then use $y$ and $z$ to restrict 
back to the upper left block.

\subsection{Notations} \label{sec:notations}

Throughout the paper, we will assume that we have fixed $\nn$ and $\kk$ and 
will use the following notational 
conventions (some 
of which are may or may not be ``conventional''):
\begin{itemize}
\item For an $n \times n$ Hermitian matrix $M$, we will write its (real) 
eigenvalues as 
\[
\lambda_1(M) \geq \lambda_2(M) \geq \dots \geq \lambda_n(M)
\]
and its characteristic polynomial as 
\[
\charpoly{M} = \mydet{x I - M}.
\]
\item For a degree $d$ real rooted polynomial $p$, we will write its roots as
\[
\mu_1(p) \geq \mu_2(p) \geq \dots \geq \mu_d(p)
\]
\item For a matrix $C$ of size $\mm \times \nn$ and sets $S \subseteq[m]$ and 
$T \subseteq [n]$ with $|S|=|T|$, denote by $[C]_{S,T}$ the {\em $(S, 
T)$-minor} of $C$:
\[
\minor{C}{S}{T} = \mydet{\{ C_{ij} \}_{i \in S, j \in T}}
\]
\item For integers $r>0$, we will consider the normalized all ones vectors: 
$\one_r= 
(\frac{1}{\sqrt{r}},\frac{1}{\sqrt{r}}\dots,\frac{1}{\sqrt{r}} )^T$ $\in 
\mathbb{R}^r$
\item We will denote  the free sum of two matrices  by : $ A\oplus B := 
\diag{A}{B}$
\item $\orth_{\nn}$ will denote the group of $n \times n$ orthogonal matrices, 
$\perms_{\nn}$ the group of $n \times n$ permutation matrices and $\reps_{\nn}$ 
the 
standard representation of the symmetric group $S_{n+1}$.
\item For a set $X$ of real numbers, we will write
$\|X\|_1 
= 
\sum_{x \in X} x$. 

\end{itemize}
Furthermore, we will attempt to state (and prove) results that hold for general 
$\mm \times \nn$ matrices in their full generality.
This is particularly true for the results in Section~\ref{convo}, 
Section~\ref{sec:polytime} and also Proposition~\ref{prop:bound}.
However one can keep in mind that our eventual application will be the special 
case $\mm = \nn \kk$, and we will sometimes write quantities in terms of $\mm$ 
when it makes things cleaner\footnote{However, when we do, we will try to make 
this explicit as in Lemma~\ref{lem:better}.}.

\section{Convolution operations and interlacing properties}\label{convo}

\subsection{Rectangular convolution and rectangular bounds} 
To study expected characteristic polynomials of the graphs we will consider, we 
will use some new finite free convolution for singular values, whose properties 
are studied and introduced in \cite{GM1}.

Let $p$ and $q$ be degree $\nn$ polynomials with all non-negative roots:
\[
p(x)=\sum_{i=0}^{\nn}  x^{\nn-i} (-1)^ia_i \quad \text { and   }  \quad q(x)= \sum_{i=0}^{\nn} x^{\nn-i} (-1)^ib_i
\]
Then for an integer $m \geq n$ we define the {\em rectangular additive 
convolution} of $p$ 
and $q$ to be the polynomial
\[
p \recsum_{\mm,\nn} q(x) = \sum_{\ell=0}^{\nn} x^{\nn-\ell} 
(-1)^\ell\sum_{i+j=\ell} \frac{(\nn-i)! 
(\nn-j)!}{\nn!(\nn-\ell)!}\frac{(\mm-i)!(\mm-j)!}{\mm!(\mm-\ell)!}a_i b_j
\]
We note that this definition is slightly different than the one given in 
\cite{GM1}, because $m$ and $n$ have specific interpretations in this context 
(they will correspond to the dimensions of a matrix).
This was not the case in $\cite{GM1}$, where the treatment was restricted to 
polynomials.
The translation between the two definitions is
\[
p \recsum_{\mm,\nn} q(x)= p\mysum{\nn}{\mm-\nn}q.
\]
The relevance of this convolution to the current paper comes from the following 
theorem, which is a direct consequence of the ``Local'' theorem in 
\cite{new_poly}:
\begin{thm}\label{thm:convo}
If $A$ and $B$ are $m \times n$ matrices with $p(x)= \chi(A^TA), q(x)= 
\chi(B^TB)$
then
\[
p \recsum_{\mm,\nn} q(x) =
\expect{Q, R}{
\charpoly{ \left( A+ QBR^T 
\right)^T\left( A+ QBR^T\right)
}}
\]
where the expectation can be taken over any independent random matrices $Q$ and 
$R$ that are {\em minor-orthogonal}.
\end{thm}
\begin{rmk}\label{rmk:minor_orth}
We refer the reader to \cite{new_poly} for the precise definition of minor 
orthogonality --- for our purposes it is suffices to know that the following 
random $n \times n$ matrices have this property:
\begin{itemize}
\item a random $n \times n$ orthogonal matrix (under the Haar measure)
\item a random $n \times n$ signed permutation matrix (under the uniform 
measure)
\item a random member of the standard representation of $S_{n+1}$ (under the 
uniform measure).
\end{itemize}
\end{rmk}
Recalling that for bipartite matrices we have
\[
B= \offdiag{A}
\]
where $A$ is rectangular of size $\mm\times \nn$, we get the following 
corollary:
\begin{cor}
If $A$ and $B$ are $m \times n$ matrices with $p(x)= \chi(A^TA), q(x)= 
\chi(B^TB)$
then
\[
\Strans(p \recsum_{\mm,\nn} q)(x)
=  
\expectt
{Q \in O_{\mm}}
{R \in O_{\dd}}
{ 
    \charpoly{
        \offdiag{A}
        +
        \diag{Q}{R}
        \offdiag{B}
        \diag{Q}{R}^T
}}
\]
where we are using the operator on polynomials $\Strans \{p(x) \} = p(x^2)$.
\end{cor}

While the results of \cite{new_poly} do not (in general) preserve real 
rootedness, this is the case for the rectangular additive convolution, as was 
shown in \cite{GM1}:
\begin{lemma}
If $p$ and $q$ are degree $n$ polynomials with all nonnegative roots, then for 
all $m \geq n$,  $p\recsum_{\mm,\nn} q$ is a degree $n$ polynomial with all 
nonnegative roots.
In addition, the $\recsum_{\mm,\nn}$ operation is bilinear and associative.
\end{lemma}
The primary contribution of \cite{GM1}, however, is an inequality that holds 
between 
the roots of the rectangular convolution and the roots of the original 
polynomials.
To state it, we let $m \geq n$ be fixed and define the polynomial transformation
\[
V \{ p(x) \} = x^{\mm-\nn}p(x).
\]
Now for any real number $u$, we can define the quantity
\[
\mathcal{Q}^{\mm,\nn}_p (u)
= \maxroot{ (\Strans p) ( \Strans Vp) - u 
(\Strans p)' (\Strans Vp)' }.
\]
Of particular importance is the fact that, for a polynomial $p$ with 
nonnegative real roots, se have 
\[
\mathcal{Q}^{\mm,\nn}_p (0) 
= \maxroot{ (\Strans p) ( \Strans Vp)}
= \lambda_{\max}(\Strans p).
\]
The other two properties of $\mathcal{Q}^{\mm,\nn}_p (u)$ that we will need are 
stated in the following theorem, which is proved in \cite{GM1}.
\begin{thm} \label{Inmain}
For all degree $\nn$ polynomials $p$ and $q$ having only nonnegative roots, and 
for all integers $m \geq 0$ and real numbers $u \geq 0$, we have
\[
\frac{\partial}{\partial u} \mathcal{Q}^{\mm,\nn}_p (u) \geq 0
\]
and 
\[
\sqrt{\mathcal{Q}^{\mm,\nn}_{p \recsum_{\mm,\nn} q} (u^2)^2+(\mm-\nn)^2u^2} \leq \sqrt{\mathcal{Q}^{\mm,\nn}_{p} (u^2)^2+(\mm-\nn)^2u^2}+\sqrt{\mathcal{Q}^{\mm,\nn}_{q} (u^2)^2+(\mm-\nn)^2u^2} -(\mm+\nn)u
\]
\end{thm}

\subsection{Some adapted interlacing properties}\label{interlacing}
In this section, we show how to prove Proposition \ref{Ebound} from the 
interlacing properties of polynomials.
Recall the following definitions and theorems from \cite{MSS4}:
\begin{defn} [Random swap] A {\em random swap} is a matrix valued random 
variable which is equal to a transposition of two (fixed) indices i,j with 
probability $\alpha$  and equal to the identity with probability $ (1-\alpha)$ 
for some $\alpha \in [0,1]$. 
\end{defn} 
\begin{defn} [Realizibility by swaps]
A matrix-valued random variable $M$ supported on permutation matrices is said 
to be {\em realizable by swaps}  if the distribution of $M$ is the same as the 
distribution of a product of independent random swaps. 
\end{defn}
\begin{thm}\label{thm:mss4}
Let $A_1,\dots,A_{\dd}$ be $r \times r$ symmetric matrices and let $M_1, \dots, 
M_d$ be independent random permutations that are realizable by swaps.
Then
\[
\expect{}{\det \Big(xI - \sum_{i=1}^\dd M_i A_i M_i^T\Big)}
\]
is real-rooted and 
\[
\lambda_2\big(  \sum_{i=1}^\dd M_i A_i M_i^T\big) \leq \lambda_2\Big( \mathbb{E}\chi \big( ( \sum_{i=1}^\dd M_i A_i M_i^T\big)\Big)
\]
with non-zero probability.
\end{thm}

It is also shown in \cite{MSS4}[Lemma~3.5] that $P \oplus S$
is realizable by swaps when $P \in \perms_\mm$ and $S \in \perms_n$ 
are picked independently and uniformly at random.
As such, Theorem~\ref{thm:mss4} directly implies Proposition~\ref{Ebound}.

\subsection{Quadrature with permutation matrices} \label{Squad}

There is still a small issue with using the rectangular convolution in the 
context of random permutations matrices.
Recall that Theorem~\ref{thm:convo} was stated to hold for random matrices that 
were {\em minor-orthogonal} and that three such random matrices were listed in 
Remark~\ref{rmk:minor_orth}.
However, the collection of matrices in $\perms_n$ is not one of those 
(and, in fact, the theorem is not true if one tries to substitute 
$\perms_n$ into the formula.

However, a slight variant of the permutation matrices {\em is} one of the 
listed random matrices, in the form of the standard representation 
$\reps_n$.
For our purposes, it suffices to know that the matrices in $\reps_n$ can 
be formed by mapping the matrices $\perms_{n+1}$ on to the set of $n 
\times n$ matrices in a way with sends their common eigenvector (the all $1$ 
vector) to $0$.
Hence instead of applying Theorem~\ref{thm:convo} directly to the 
characteristic polynomials of our adjacency matrices, we will apply them to the 
polynomials that we get after projecting them orthogonally to the all $1$ 
vector.
Fortunately (due to the biregularity), the all $1$'s vector is a common 
eigenvector of all of the other matrices we will need to consider, and so this 
projection will simply pull off a single (trivial) root without affecting the 
remaining ones.
All of this is stated in the following theorem (recall that we are using 
$\one_\nn$ to denote the normalized version of the all $1$'s vector):
\begin{thm}\label{thm:quad}
Let $A$ and $B$ be $\mm \times \nn$ matrices for which 
\[
A \one_{\nn}=a\one_\mm 
\AND
A^T\one_\mm=a\one_\nn
\AND 
B\one_{\nn}=b\one_\mm
\AND
B^T\one_\mm=b\one_\nn
\]
so that 
\[
\charpoly{A^T A} = (x - a^2)p(x)
\AND
\charpoly{B^TB} = (x -b^2) q(x).
\]
Then
\begin{align*}
&\expectt 
{P \in \perms_\mm}
{S \in \perms_\nn}
{
    \Strans \charpoly{
        \offdiag{A}
        +
        (P \oplus S)
        \offdiag{B}
        (P \oplus S)^T
    }
}
\\=&
\Strans \left\{ \left[ x-(a+b)^2 \right] \left[p \recsum_{\mm-1,\nn-1}q \right] 
\right\}
\\=& 
\left( x^2 - (a+b)^2 \right) \mathbb{S} (p \recsum_{\mm-1,\nn-1} q)
\end{align*}
\end{thm}
\begin{proof}
We begin by applying change of basis to $A$ and $B$ related to their singular 
value decompositions. 
More precisely, let $U \in \orth_\mm, V \in \orth_\nn$, such that the last 
column of $U$ 
is  $\one_\mm$ and the last vector of $V$ is  $\one_\nn$. 
As for any $P\in 
\perms_\mm, S \in \perms_\nn$ : $P\one_\mm=\one_\mm, S \one_\nn=\one_\nn$, we 
get 
\[
U^TAV=\hat{A} \oplus a 
\qquad  
U^TBV=\hat{B} \oplus b 
\qquad  
U^TPU=\hat{P} \oplus 1
\qquad
V^TSV=\hat{S} \oplus 1
\]
where the matrices $\hat{P}$ and $\hat{S}$ are random elements of 
$\reps_{\mm-1}$ and $\reps_{\nn-1}$ respectively.
Conjugating by $U \oplus V$ and using the invariance of the determinant by 
change of basis we get:
\begin{align*}
&
\charpoly{
\offdiag{A}
+
(P \oplus S)
\offdiag{B}
(P \oplus S)^T
}
\\=&
\charpoly{
\offdiag{(\hat{A} \oplus a)} +
(\hat{P} \oplus 1) \oplus (\hat{S} \oplus 1)
\offdiag{(\hat{B} \oplus b)}
(\hat{P} \oplus 1) \oplus (\hat{S} \oplus 1)^T}
\\ = &
\charpoly{\offdiag{ \big( (\hat{A} + \hat{P}\hat{B}\hat{S}^T)\oplus (a+b) 
\big)}}
\end{align*}
So averaging gives
\begin{align*}
& 
\expectt{P \in \perms_m}{S \in \perms_n}{
\Strans
\charpoly{
\left( (\hat{A} + \hat{P}\hat{B}\hat{S}^T )\oplus (a+b)\right) \big(   
  (\hat{A} + \hat{P}\hat{B}\hat{S}^T)\oplus (a+b) \big)^T
}}
\\=&
\left(x^2-(a+b)^2\right)  
\expectt{P \in \perms_\mm}{S \in \perms_\nn}{ 
\Strans
\charpoly{
\left( \hat{A} + \hat{P}\hat{B}\hat{S}^T) \right) \left(   
\hat{A} + \hat{P}\hat{B}\hat{S}^T \right)^T
}}
\\=& 
\left( x^2 - (a+b)^2\right) 
\expectt{P \in \perms_\mm}{S \in \perms_\nn}{ 
\charpoly{
    \offdiag{\hat{A}} 
    +
    (\hat{P} \oplus \hat{S})
    \offdiag{\hat{B}}
    (\hat{P} \oplus \hat{S})^T
}}
\\=&
\left( x^2 - (a+b)^2\right) 
\expectt{\hat{P} \in \reps_{\mm-1}}{\hat{S} \in \reps_{\nn-1}}{ 
\charpoly{
    \offdiag{\hat{A}} 
    +
    (\hat{P} \oplus \hat{S})
    \offdiag{\hat{B}}
    (\hat{P} \oplus \hat{S})^T
}}
\\=&
\left( x^2 - (a+b)^2\right) (p \recsum_{\mm-1,\nn-1} q)
\end{align*}
where the last equality comes from the fact that $p(x)= 
\charpoly{\hat{A}^T\hat{A}}$ and $q(x)=\charpoly{\hat{B}^T\hat{B}}$ and that 
$\reps_\mm$ and $\reps_\nn$ are minor-orthogonal and Theorem~\ref{thm:convo}.
%
\end{proof}
Applying Theorem~\ref{thm:quad} inductively gives the following corollary:
\begin{cor} \label{Ecor}
Let $A_1,\dots,A_k$ be $\mm \times \nn $ matrices for which 
\[
A_i \one_{\nn}= a_i \one_\mm
\AND
{A_i}^T \one_\mm= a_i \one_\nn
\AND
\charpoly{A_i^TA_i} 
= (x-a_i^2)p_i(x)
\]
for all  $i \in [k]$.
Then
\[
\expectt
{P_1,\dots,P_k}
{S_1,\dots,S_k}
{ 
\charpoly{
\sum_{i=1}^k 
    (P_i \oplus S_i) 
    \offdiag{A_i}
    (P_i \oplus S_i)^T
}}
= 
\left(x^2- \left(\sum_{i=1}^la_i\right)^2 \right) \mathbb{S} (p_1 
\recsum 
\dots \recsum p_k)
\]
where $\recsum$ here is short for $\recsum_{\mm-1, \nn-1}$.
\end{cor}

\section{Ramanujan Bound} \label{bound}

In this section we use the results of the previous section to get the Ramanujan 
bound (proving Theorem \ref{thm:main}). 
We recall that the 
adjacency matrix of any $(n, k, d)$-graph has the form:
\[
B = \sum_{i=1}^\dd   
(P_i \oplus S_i) 
\offdiag{I_\nn^{[\kk]}}
(P_i \oplus S_i )^T 
\qquad\text{where}\qquad   
I_\nn^{[\kk]}
= \begin{pmatrix}I_\nn\\I_\nn\\\vdots\\I_\nn \end{pmatrix}
\]
for some permutation matrices $P_i$ and $S_i$.
We first note that $B$ satisfies the eigenvector properties needed to use 
Corollary~\ref{Ecor}.
In particular, we have for $A = I_\nn^{[\kk]}$ and $m = k n$, 
\[
A\one_{\nn}= \sqrt{\frac{\mm}{\nn}}\one_{\mm}= \sqrt{\kk} \one_{\mm}
\AND
A^T\one_{\mm}=\kk \sqrt{\frac{\nn}{\mm}} \one_{\nn}= \sqrt{\kk} \one_n
\]
and 
\[
\charpoly{A^T A} = (x - k)^n = (x - \sqrt{k}^2)(x-k)^{n-1}
\]
Hence by Corollary~\ref{Ecor} we have:
\begin{align*}
\expect{}{
\charpoly{B}
}
&= \left(x^2- \left(\sum_{i=1}^\dd \sqrt{\kk}\right)^2\right) 
\mathbb{S} (p \recsum_{\mm-1,\nn-1} \dots \recsum_{\mm-1,\nn-1}p)  \quad \text{ 
($\dd$ times)}
\\&=  
(x- \dd\sqrt{\kk}) (x+\dd\sqrt{\kk})\mathbb{S} (p 
\recsum_{\mm-1,\nn-1} \dots \recsum_{\mm-1,\nn-1}p) 
\end{align*}
for $p(x):= (x-\kk)^{\nn-1}$. 
As the first two factors are the ``trivial'' eigenvalues, we are then 
interested in the largest root of the polynomial
\begin{equation}\label{eq:goal}
\mathbb{S} (p \recsum_{\mm-1,\nn-1}\dots\recsum_{\mm-1,\nn-1}p). 
\end{equation}
To this end, we consider the polynomial
\begin{equation}\label{eq:qpoly}
q_\dd^{[\theta,\mm,\nn]} 
:= \mathbb{S} \underbrace{\left(  
(x-\theta)^\nn\recsum_{\mm,\nn} 
(x-\theta)^\nn
\recsum_{\mm,\nn} \dots\recsum_{\mm,\nn}(x-\theta)^\nn \right) }_{ \dd 
\text{ times} },
\end{equation}
For the sake of generality, we prove the following proposition for general 
values of $\mm \geq n$ here (not $\mm = \kk \nn$) as in the rest of this 
section.

\begin{prop}\label{prop:bound}
For all $u, \theta \geq 0$ and all integers $\mm, \nn, d$ with $\mm \geq \nn$
we have
\[
\lambda_1\left(q_\dd^{[\theta,\mm,\nn]}\right) \leq \Big(\dd \sqrt{\theta + 
\mm^2u^2}- \dd \mm u+(\mm+\nn)u \Big)^2- (\mm-\nn)^2u^2=: 
R^{[\theta,\mm,\nn]}_\dd(u)
\]
\end{prop}

\begin{proof}
By the first part of Theorem~\ref{Inmain}, we can directly upper-bound for all 
$u \geq 0$:
\[
\lambda_1\left(q_\dd^{[\theta,\mm,\nn]}\right) \leq 
\mathcal{Q}^{\mm,\nn}_{q_\dd} 
(u^2):={t_\dd^{[\theta,\mm,\nn]}}.
\]
Furthermore, by the second part of Theorem~\ref{Inmain}, we have
\begin{align}
\sqrt{{t_\dd^{[\theta,\mm,\nn]}}^2+(\mm-\nn)^2u^2} 
&\leq 
\sqrt{{t_{\dd-1}^{[\theta,\mm,\nn]}}^2+(\mm-\nn)^2u^2}+\sqrt{{t_1^{[\theta,\mm,\nn]}}^2+(\mm-\nn)^2u^2}
 -(\mm+\nn)u \notag
\\&\leq 
\dd \sqrt{{t_1^{[\theta,\mm,\nn]}}^2+(\mm-\nn)^2u^2} -(\dd-1)(\mm+\nn)u 
\label{eq:td}.
\end{align}
We claim that 
\begin{equation}\label{eq:t1}
\sqrt{({ t_1^{[\theta,\mm,\nn]}}^2 +(\mm-\nn)^2u^2 } = \nn u + \sqrt{\theta + 
\mm^2 u^2}
\end{equation}
which, if it holds, can be plugged into \eqref{eq:td} to imply that
\[
\sqrt{{t_\dd^{[\theta,\mm,\nn]}}^2+(\mm-\nn)^2u^2}
\leq 
\dd \sqrt{\theta + \mm^2 u^2} - \dd \mm u + (\mm+\nn) u
\]
for any value of $u \geq 0$ (hence proving the proposition).
To see that \eqref{eq:t1} holds, we note that (by definition)
\[
{t_1^{[\theta,\mm,\nn]}}^2 = \maxroot{(x^2 - \theta)^2 - 4 \mm \nn u^2 (x^2 - \theta) - 4 \nn^2 \theta u^2}
\]
which, by the quadratic equation gives
\[
{t_1^{[\theta,\mm,\nn]}}^2 - \theta 
= \maxroot{x^2 - 4 \mm \nn u^2 x - 4 \nn^2 \theta u^2}
= 2 \nn \mm u^2 + 2\nn u\sqrt{\theta + \mm^2u^2}.
\]
Hence we have
\begin{align*}
{t_1^{[\theta,\mm,\nn]}}^2 + (\mm-\nn)^2u^2 
&= \theta + 2 \nn \mm u^2 + 2 \nn u\sqrt{\theta + \mm^2 u^2} + (\mm-\nn)^2 u^2
\\&= \theta + \mm^2 u^2 + \nn^2 u^2 +2 \nn u\sqrt{\theta + \mm^2 u^2}
\\&= \left( \nn u + \sqrt{\theta + \mm^2 u^2}\right)^2
\end{align*}
which in turn proves \eqref{eq:t1}.
\end{proof}

Since Proposition~\ref{prop:bound} holds for any $u \geq 0$, we are free to 
choose the best one (which will obviously depend on the value of $\theta$).
For $\theta \geq 2$, simple calculus provides us with the value
\[
u_0 = \frac{\sqrt{\theta}(v-1)}{2 \mm \sqrt{v}}.
\]
where (for the sake of clarity) we have used the substitution $v = 
\sqrt{\dd-1}\sqrt{\dd 
\mm/\nn- 
1}$.
Plugging this into Proposition~\ref{prop:bound}, then gives the following:
\begin{cor}\label{cor:OK}
For all $\theta \geq 2$ and all positive integers $d, n$ and $m \geq n$, we have
\[
\lambda_1\left(q_\dd^{[\theta,\mm,\nn]}\right) 
\leq 
\sqrt{\frac{\theta\nn}{\mm}} \left(\sqrt{\dd-1} + \sqrt{(\dd \mm/\nn- 
1}\right).
\]
\end{cor}
At first glance, Corollary~\ref{cor:OK} seems to be precisely the bound we want 
(recalling that we wish to set $m = k n$).
However, one will notice that our definition of $q_d^{[\theta, m, n]}$ in 
\eqref{eq:qpoly} is not exactly the quantity in \eqref{eq:goal} that we want to 
find (it made the math a bit simpler to write).
For the quantity we are actually interested in, Corollary~\ref{cor:OK} implies 
the bound
\[
 \lambda_1\left(q_\dd^{[\kk,\mm-1,\nn-1]}\right)  \leq 
 \sqrt{\frac{\kk(\nn-1)}{\mm-1}} \left(\sqrt{\dd-1} + 
 \sqrt{\frac{\mm-1}{\nn-1}\dd - 1}\right).
 \]
Fortunately for us (and curiously), the bound we get because of this shift is 
actually {\em better} than Ramanujan. 
\begin{lemma}\label{lem:better}
For all $k > 0$, $\dd\geq 2$, $n \geq 1$ and $m = k n$, we have
\begin{align*}
\sqrt{\frac{\kk(\nn-1)}{\mm-1}} \left(\sqrt{\dd-1} + 
\sqrt{\frac{\mm-1}{\nn-1}\dd - 1}\right) 
&=  \sqrt{(\dd-1)}\sqrt{\left(1 - \frac{\mm-\nn}{\nn(\mm-1)}\right)} + 
\sqrt{\dd \kk - 
1 + \frac{\mm-\nn}{\nn(\mm-1)}}\\
&\leq \left(\sqrt{\dd-1} + \sqrt{\dd\kk - 1}\right)
\end{align*}
\end{lemma}
\begin{proof}
Consider the function 
\[
f(x) = a\sqrt{1-x} + \sqrt{b+x}
\]
for positive numbers $a, b$.
Then 
\[
f'(x) = \frac{\sqrt{1-x} - a\sqrt{b+x}}{2\sqrt{1-x}\sqrt{b+x}}
\]
and so $f'(x) \leq 0$ for $x \in [0, 1]$ and $a, b \geq 1$.
Hence
\[
a \leftarrow \sqrt{\dd-1} \geq 1
\AND
b \leftarrow \dd\kk-1 \geq 1
\AND
x \leftarrow \frac{m-n}{n(m-1)} < 1
\]
implies the inequality we want (which in particular is true for $\dd \geq 2$).
\end{proof}
 
\section{Polynomial time algorithm} \label{poly}

\newcommand{\Cl}{C_\ell}
\newcommand{\El}{E_\ell}
\newcommand{\hCl}{\hat{C}_\ell}
\newcommand{\ess}{t}
\newcommand{\hA}{\hat{A}}
\newcommand{\hP}{\hat{P}}
\newcommand{\hS}{\hat{S}}
\newcommand{\tA}{\tilde{A}}

\newcommand{\node}{w}

\subsection{Description of the algorithm}\label{sec:alg}

The algorithm follows the proof of existence, starting with the expected 
characteristic polynomial and using an interlacing family to find a particular 
polynomial with 
see \cite{MCohen} and \cite{MSS1} for more details). 
We will consider a tree with many nodes, where each node $\node$ corresponds to 
a 
graph $G_\node$.
The roots node corresponds to the empty graph on $kn + n$ vertices, and the 
graphs corresponding to child nodes will consist of the graph corresponding to 
the parent node with a single edge added.
This will be  done in a particular way so that after every $k$ edges we have 
added another $k$-claw and after every $nk$ edges, we have added another 
$k$-claw matching.
Intuitively, one can think of the graphs $G_\node$ as ``partial assignment'' 
graphs;  that is, we construct the Ramanujan graph one edge at a time and 
$G_\node$ 
is the graph consisting of edges we have picked so far.

To each node $n$ we also associate a polynomial $p_\node$, however these will 
be 
built starting at the leaf nodes and working up the tree.
The polynomials $p_\node$ associated to a leaf node $\node$ will simply be the 
characteristic polynomial of the adjacency 
matrix of $G_\node$ (its associated graph).
All other nodes will receive the polynomial formed by averaging over all of the 
leaf node polynomials below it in the tree.
Given the interpretation of $G_\node$ as a partial assignment graph, the 
resulting 
polynomials can be thought of as the expected characteristic polynomial of the 
final graph conditioned on the fact that it must contain $G_\node$ as a 
subgraph.
We will perform a greedy search of the graph, moving $k$ levels at a time (so 
corresponding to adding a new $k$-claw matching each turn), starting with the 
root and descending all the way to one of the leaves (which will be the 
resulting Ramanujan graph).

Of particular note is that the root (which, as we noted above, corresponds to 
the empty graph) will have the expected characteristic polynomial (taken over 
all of the leaves) associated to it, which (by Theorem~\ref{thm:main}) has its 
second largest root smaller than the Ramanujan bound.
Furthermore, this polynomial can be written in the form
\[
\expectt
{P_1, \dots,P_{\dd} \in \perms_{\nn\kk}}
{S_1,\dots,S_{\dd} \in \perms_{\nn}} 
{
\charpoly{
    \sum_{i=1}^{\dd} 
    (P_i \oplus S_i)  
    \offdiag{I_{\nn}^{[\kk]}}
    (P_i \oplus S_i )^T
    }
}.
\]
Note that each pair of permutations $(P_i, S_i)$ corresponds to a $k$-claw 
matching, with the sum being the union of the graphs. 
The leaves of our tree will include all possible biregular bipartite graphs 
that can be formed this way (with repetitions) and their associated polynomials 
will be
\[
\charpoly{
    \sum_{i=1}^{\dd}   
    (P_i \oplus S_i) 
    \offdiag{I_{\nn}^{[\kk]}}
    (P_i \oplus S_i )^T
}
\]
for each possible sequence of permutation matrices $P_1, 
\dots,P_{\dd} \in \perms_{\nn \kk}$ and $S_1,\dots S_{\dd} \in \perms_\nn$.

By construction, every node will differ by exactly one edge (the same edge) 
from each of its siblings.
The utility of this is that the characteristic polynomials of the adjacency 
matrices of such graphs will have a common interlacer.
This makes the tree itself an {\em interlacing family} (as first defined in 
\cite{MSS1}) and therefore gives it the property that every parent node has at 
least one child whose associated polynomial has a smaller $k$th root than the 
one associated to the parent.
Hence any bound on the second eigenvalue of a node can be propagated down the 
tree until one hits a leaf.
By Theorem~\ref{thm:main} and Lemma~\ref{lem:better}, we know that the root 
satisfies the Ramanujan bound, and so there exists a path down the tree which 
maintains this bound.
Every node will have at most $kn$ children, and so the number of nodes that are 
$k$ levels lower will be at most $(kn)^k$ (which, for fixed $k$, is polynomial 
in $n$).
Hence we can examine all of them and find the one whose associated polynomial 
has the smallest second eigenvalue, if we can compute the associated 
polynomials efficiently.
This is precisely what was shown in \cite{MCohen} for the special case $k = 1$.

The first observation one can make when trying to construct the associated 
polynomials is that the polynomials associated to nodes that are on levels 
which are multiples of $kn$ are easy to compute.
This is because the corresponding graph $G_n$ contains a union of complete 
$k$-claw matchings, and so the expectation is over the remaining $k$-claw 
matchings that need to be placed.
In particular, this causes the polynomial to be an expectation over entire 
permutation groups and this provides enough symmetry to allow us to compute the 
associated polynomial for this node from the characteristic polynomial of its 
associated graph using polynomial convolutions (see, for example, \cite{HPS}).
Hence we can restrict our attention to graphs where $r-1$ $\kk$-claw matchings 
have been chosen and the $r$th $\kk$-claw matching has been partial assigned.
Since we are moving down $k$ rows at a time, we will always be in a situation 
where $\ess$ of the $\kk$-claws have been chosen and therefore $\ell$ more 
$\kk$-claws remain to finish the $r$th $\kk$-claw matching (so $\ess + \ell = 
\nn$).
If we let
\[
\offdiag{A}
\]
be the adjacency matrix of the graph at this node (where $A$ is a $\kk \nn 
\times \nn$ rectangular matrix), we can then represent the 
$\ell$ remaining $\kk$-claws by the block matrix 
\[
\Cl= 
\begin{pmatrix}
I_\ell^{[k]} & 0_{\kk \ell \times \ess} \\
0_{ \kk \ess \times \ell}&0_{\kk \ess \times \ess}
\end{pmatrix}
\]
(where $I_\ell^{[k]}$ consists of $\kk$ vertical copies of the identity matrix 
of size 
$l$).
Hence the entire algorithm comes down to our ability to compute (in polynomial 
time) the following quantity:
\begin{equation}\label{eq:alg_goal}
\expectt
{P \in \mathcal{P}_{\kk \ell}}
{S \in \mathcal{P}_{l}}
{
    \charpoly{
        \offdiag{A}
        +
        \diag{P \oplus I_{\kk\ess}}{S \oplus I_{\ess}}
        \offdiag{\Cl}
        \diag{P \oplus I_{\kk\ess}}{S \oplus I_{\ess}}^T
    }
}
\end{equation}

\subsection{Elimination of the second permutation}

The first issue that one finds when trying to compute \eqref{eq:alg_goal} is 
the appearance of two permutation matrices $P$ and $S$.
However we will see that computing all possible permutations of 
$I_\ell^{[\kk]}$ in this manner is redundant --- in fact, we will be able to 
replace the average in \eqref{eq:alg_goal} (which contains both $P$ and $S$) by 
an average over a single rectangular matrix in a sufficiently regular group. 
To add complication to the issue, all of this must be done after we have 
projected out the all $1$ vector (so that we can use Theorem~\ref{thm:quad}).

\begin{rmk}
The primary reason for this redundancy (at least in linear algebraic terms) is 
the fact that $\Cl$ behaves like a projection matrix.  
The clearest way to see this (and the approach we will take), 
is to ``diagonalize'' $\Cl^-$ using the singular value decomposition. 
The case when $\kk=1$, however, is much less cumbersome because in that case 
$\Cl$ is already in a diagonal form.
As such, the reader will note that much of this subsection was reduced to a 
passing remark in \cite{MCohen}.
\end{rmk}

To begin, let $\El$ be the block matrix
\[
\El= \frac{1}{\ell}
\begin{pmatrix}
J_{\kk \ell \times \ell} &0_{\kk \ell \times \ess} \\
 0_{k \ess  \times \ell}&0_{\kk \ess \times \ess}
\end{pmatrix}
\]
where $J_{\kk \ell \times \ell}= \one_{\kk \ell} \one_\ell^{T}$ is the $\kk 
\ell 
\times \ell$ rectangular matrix with only ones. 
Furthermore, let
\[
A^+ = A + \El
\AND
\Cl^- = \Cl - \El.
\] 
Since $J$ is invariant with respect to permutations, it suffices to consider:
\[
\expectt
{P \in \mathcal{P}_{\kk \ell}}
{S \in \mathcal{P}_{l}} 
{
    \charpoly{
        \offdiag{A^+}
        +
        \diag{P \oplus I}{S \oplus I}
        \offdiag{\Cl^-}
        \diag{P \oplus I}{S \oplus I}^T
    }
}
\] 
We start by changing the basis to isolate the vector $\one$ which is a singular 
vector (on both sides) for all of these matrices.
Let
\[
\Cl^-
= 
(U \oplus I_{\kk \ess})^T \hCl (V \oplus I_{\ess}). 
\]
be the singular decomposition of $\Cl^-$ for which the $\ell$th singular vector 
is the one corresponding to $\one$.
In particular,
\[
\hCl 
= \begin{pmatrix}
\sqrt{\kk} I_{\ell-1} \oplus 0 & 0_{\ell \times \ess} \\
0_{(\kk \nn - \ell) \times \ell} & 0_{\kk \nn - \ell \times \ess}
\end{pmatrix}
\]
and $U \in \orth_{\kk\ell}$ and $V \in \orth_{\ell}$.
Hence we have
\begin{itemize}
\item $(\hP \oplus 1) \oplus I_{\kk\ess} = (U \oplus I_{\kk \ess}) (P \oplus 
I_{\kk \ess}) (U \oplus I_{\kk \ess})^T$, and
\item $(\hS \oplus 1) \oplus I_{\ess} = (V \oplus I_{\ess}) (S \oplus 
I_{\kk \ess}) (V \oplus I_{\ess})^T$.
\end{itemize}
where $\hP \in \reps_{\kk\ell-1}$ and $\hS \in \reps_{\ell-1}$.
Finally, if we set
\[
\hA = 
(U \oplus I_{\kk \ess}) A^+ (V \oplus I_{\ess})^T
\]
then by the invariance of the determinant with respect to change of basis we 
have that \eqref{eq:alg_goal} is the same as
\begin{equation}\label{eq:alg_goal2}
\expectt
{\hP \in \reps_{\kk \ell -1}}
{\hS \in \reps_{\ell-1}}
{
    \charpoly{
        \offdiag{\hA} + 
        \diag{\hP \oplus I_{\kk\ess+1}}{\hS \oplus I_{\ess+1}}
        \offdiag{\hCl}
        \diag{\hP \oplus I_{\kk\ess+1}}{\hS \oplus I_{\ess+1}}^T
    }
}.
\end{equation}

The main results of \cite{new_poly} or \cite{HPS} imply the 
following extension of Corollary 4.12 from \cite{MSS4}: 
\begin{lemma}\label{lem:quad}
If $C$ and $D$ are $(\mm+\nn)\times(\mm+\nn)$ symmetric matrices and $r <\mm$  
then
\[
\expect
{\hP \in \reps_{r}}{
    \charpoly{
        C+ 
        (\hP \oplus I)
        D
        (\hP \oplus I)^T
}}
= 
\expect
{Q \in \orth_{r}}{
    \charpoly{
        C+ 
        (Q \oplus I)
        D
        (Q \oplus I)^T
}}.
\]
where each $I = I_{m+n-r}$.
\end{lemma}
Decomposing the matrix 
\[
\diag{\hP \oplus I_{\kk\ess+1}}{\hS \oplus I_{\ess+1}}
= 
\diag{\hP \oplus I_{\kk\ess+1}}{I_{\nn}}
\diag{I_{\kk\nn}}{\hS \oplus I_{\ess+1}}
\]
we can apply Lemma~\ref{lem:quad} twice to \eqref{eq:alg_goal2} to get
\begin{align}
&\expectt
{\hP \in \reps_{\kk \ell -1}}
{\hS \in \reps_{\ell-1}}
{
    \charpoly{
        \offdiag{\hA} + 
        \diag{\hP \oplus I_{\kk\ess+1}}{\hS \oplus I_{\ess+1}}
        \offdiag{\hCl}
        \diag{\hP \oplus I_{\kk\ess+1}}{\hS \oplus I_{\ess+1}}^T
    }
}
\notag\\=& \notag
\expectt
{Q \in \orth_{\kk \ell -1}}
{R \in \orth_{\ell-1}}
{
    \charpoly{
        \offdiag{\hA} + 
        \diag{Q \oplus I_{\kk\ess+1}}{R \oplus I_{\ess+1}}
        \offdiag{\hCl}
        \diag{Q \oplus I_{\kk\ess+1}}{R \oplus I_{\ess+1}}^T
    }
}
\notag\\=&
\expectt
{Q \in \orth_{\kk \ell-1}}
{R \in \orth_{\ell-1}}
{
    \charpoly{
        \offdiag{(\hA+(Q\oplus I)\hCl(R\oplus I)^{T})}
    }
}
\label{eq:alg_goal3}
\end{align}
Now, notice that
\[
(Q\oplus I)\hCl= \sqrt{\kk} 
\twobytwo
{Q_{(\ell-1)}} { 0_{(\ell-1) \times (\ess+1)}}
{0_{(\kk \ess +1) \times (\ell-1)}} {0_{(\kk\ess+1) \times (\ess+1) }}.
\]
where $Q_{(\ell-1)} $ corresponds to the rectangular matrix of size $(\kk 
\ell-1)\times(\ell-1)$ made up by keeping only the first $\ell-1$ columns of 
$Q$. 

\begin{rmk}\label{rmk:stiefel}
The matrix $Q_{(\ell-1)}$ is often referred to as an $(\ell-1)$-frame of $Q$.
More generally for $s>r$, 
the set of orthonormal $r$-frames of dimension $s$ consist of all rectangular 
$s \times r$ matrices such that the columns form an orthonormal family of 
vectors in $\mathbb{R}^s$.
Such families are well studied (see, for example, \cite{stiefel}).
In particular, it is well known that the set of 
$r$-frames of dimension $s$ forms a 
compact manifold called the {\em Stiefel manifold}, which we denote 
$\rframe{r}{s}$. 
These manifolds can be equipped with a Haar uniform measure that 
is invariant under left action by $\orth_s$ and right action by 
$\orth_r$.
Hence it follows that for any $R \in \orth_{\ell-1}$ we have
\begin{equation}\label{eq:same}
\expect
{\tilde{Q} \in V_{\ell-1}(\mathbb{R}^{\kk \ell-1})}
{f(\tilde{Q}R)}
= 
\expect
{\tilde{Q} \in V_{\ell-1}(\mathbb{R}^{\kk \ell-1})}
{f(\tilde{Q})}
\end{equation}
for any continuous function $f$, where the expectation is taken over the 
uniform/Haar measure (and similarly for multiplication on the left).
\end{rmk}
Using these properties, we can simplify \eqref{eq:alg_goal3} even further:
\begin{align}
&\expectt
{Q \in \orth_{\kk \ell-1}}
{R \in \orth_{\ell-1}} 
{ 
\charpoly{
    \offdiag{(\hA+(\sqrt{k}Q_{\ell-1}R^{T} \oplus 0))}
}}
\notag\\=&
\expect
{R \in \orth_{\ell-1}}{ 
\expect
{Q \in \orth_{\kk \ell-1}}
{ 
\charpoly{
    \offdiag{(\hA+(\sqrt{k}Q_{\ell-1}R^{T} \oplus 0))}
}}}
\notag\\=&
\expect
{R \in \orth_{\ell-1}}{ 
\expect
{\tilde{Q} \in V_{\ell-1}(\mathbb{R}^{\kk \ell-1})}
{ 
\charpoly{
    \offdiag{(\hA+(\sqrt{k}\tilde{Q} \oplus 0))}
}}}
\notag\\=&
\expect
{\tilde{Q} \in V_{\ell-1}(\mathbb{R}^{\kk \ell-1})}
{ 
\charpoly{
    \offdiag{(\hA+(\sqrt{k}\tilde{Q} \oplus 0))}
}}
\notag\\=&
\expect
{\tilde{Q} \in V_{\ell-1}(\mathbb{R}^{\kk \ell-1})}
{ 
\mathbb{S}
\charpoly{(\hA+(\sqrt{k}\tilde{Q} \oplus 0))(\hA+(\sqrt{k}\tilde{Q} 
\oplus 0))^T}.
}\label{eq:alg_goal_final}
\end{align}

\subsection{Characteristic polynomials in polynomial time}\label{sec:polytime}

The goal is to express \eqref{eq:alg_goal_final} in a computable way, which we 
do by computing an explicit formula for the coefficients in terms of $A$.
To this end, we will fix a matrix $A \in \mathbb{R}^{\mm\times \nn}$ with $m 
\geq n$ and we will let $Q$ be a matrix sampled uniformly from $\rframe{r}{s}$ 
(where $m \geq s \geq r$ and $n \geq r$)
and then consider the expansion of the polynomial 
\begin{equation}\label{eq:expect}
\expect
{Q}
{
\mydet{ xI 
        + 
        \left(A+\sqrt{k}\topleft{Q}\right) 
        \left(A+\sqrt{k}\topleft{Q}\right)^T
}
}
= \sum_{i=0}^{\nn}x^{\nn-i} c_i(A).
\end{equation}
We remind the reader that we use the notation 
\[
\minor{C}{S}{T} = \mydet{\{ C_{ij} \}_{i \in S, j \in T}}
\]
to denote the $(S, T)$-minor of $C$ and note that the main property of $Q$ that 
we will use is that it is {\em minor-orthogonal} as defined in 
\cite{new_poly}.  
That is, for any sets $S, T, U, V$ with $|S|=|T| = k$ and $|U|=|V|=\ell$, we 
have
\begin{equation}\label{minor-orth}
\expect{Q}{ \minor{Q}{S}{T}\minor{Q}{U}{V}} = \frac{1}{\binom{s}{k}} 
\test{S=U}\test{T=V}
\end{equation}
The proof of this fact is given in Lemma~\ref{frameLA} in the Appendix.

\newcommand{\sets}[5]{\substack{
#1 \in \binom{[\mm]}{#3}, #2 \in \binom{[\nn]}{#3} \\
|#1 \cap [s]|=#4, |#2 \cap [r]|=#5 
}}

\begin{prop}\label{prop:decompose}
For fixed $\mm, \nn, s, r$, the coefficients $c_i(A)$ from \eqref{eq:expect} 
have the explicit formula
\[
c_i(A) = 
\sum_{j, p, q} \frac{\binom{s-p}{i-j} 
\binom{r-q}{i-j}}{\binom{s}{i-j}}  k^{i-j}A_{p,q}^j
\quad
\text{where}
\quad
A_{p,q}^j= \sum_{
\substack{
X \in \binom{[\mm]}{j}, Y \in \binom{[\nn]}{j} \\
|X \cap [s]|=p, |Y \cap [r]|=q 
}} \minor{A}{X}{Y}^2
\]

%
%

\end{prop}
\begin{proof}

Using Corollary~\ref{coeff} and then Lemma~\ref{cb} we can start by writing
\[
c_i(A)=\sum_{S \in \binom{[\mm]}{i},T \in \binom{[\nn]}{i}}  
\minor{A +\sqrt{k} \topleft{Q}}{S}{T}^2.
\]
By Lemma~\ref{sumdet}, we can then expand
\begin{align*}
\minor{A + \topleft{\sqrt{k} Q} }{S}{T} 
&=\sum_{|U|=|V|} 
(-1)^{\|U+V\|_1} k^{|U|/2} 
\minor{ \topleft{Q} } {U}{V} \minor{A}{S \setminus U}{T \setminus 
V}\test{U \subseteq S}\test{V \subseteq T}
\\&=
\sum_{|U| = |V|} 
(-1)^{\|U+V\|_1} k^{|U|/2} 
\minor{ Q } {U}{V} \minor{A}{S \setminus U}{T \setminus V}\test{U \subseteq 
S \cap [s]}\test{V \subseteq T \cap [r]}.
\end{align*}
By minor orthogonality, we have 
\[
\expect{Q}{\minor{ Q } {U}{V}\minor{ Q } {U'}{V'}} = \frac{1}{\binom{s}{|U|} }
\test{U=U'}\test{V=V'}
\]
and so conditioning on $|U| = |V| = j$, we get
\[
\expect{Q}{
\minor{A + \topleft{\sqrt{k} Q}}{S}{T}^2
}
= 
\sum_j \frac{k^{j}}{\binom{s}{j}}
\sum_{|U|=|V|=j}
\minor{A}{S \setminus U}{T \setminus V}^2\test{U \subseteq S}\test{V \subseteq 
T}\test{U \subseteq [s]}\test{V \subseteq [r]}.
\]
Summing this over all $S$ and $T$ of size $i$, we then get
\begin{align*}
c_i(A)
&= 
\sum_{|S|=|T|=i} 
\sum_j \frac{k^{j}}{\binom{s}{j}}
\sum_{|U|=|V|=j}
\minor{A}{S \setminus U}{T \setminus V}^2\test{U \subseteq S}\test{V \subseteq 
T}\test{U \subseteq [s]}\test{V \subseteq [r]}.
\end{align*}
Now we simply need to change our reference sets: given that $U \subseteq S$ 
with 
$|U|=j$ and $|S|=i$, we must have $|S 
\setminus U| = i -j$ (and similarly with $T$ and $V$).
So if we substitute $X = S \setminus U$ and $Y = T \setminus V$, we can rewrite 
this sum as
\begin{align*}
c_i(A)
&=
\sum_j \frac{k^{j}}{\binom{s}{j}}
\sum_{|U|=|V|=j}
\sum_{|X|=|Y|=i-j} 
\minor{A}{X}{Y}^2\test{X \cap U = \varnothing}\test{Y \cap V = 
\varnothing}\test{U \subseteq [s]}\test{V \subseteq [r]}.
\\&= 
\sum_j \frac{k^{j}}{\binom{s}{j}}
\sum_{|X|=|Y|=i-j} 
\minor{A}{X}{Y}^2 \sum_{|U|=|V|=j}
\test{X \cap U = \varnothing}\test{Y \cap V = 
\varnothing}\test{U \subseteq [s]}\test{V \subseteq [r]}.
\end{align*}
Given a fixed $X$ with $|X \cap [s]| = p$, it is easy to see that 
\[
\sum_{|U|=j}\test{X \cap U = \varnothing}\test{U \subseteq [s]} = \binom{s-p}{j}
\]
and similarly for fixed $Y$ with $|Y \cap [r]| = q$.
Hence we have
\[
c_i(A)
=
\sum_{j, p, q} \frac{k^{j}}{\binom{s}{j}}
\binom{s-p}{j} \binom{r-q}{j}
\sum_{\substack{|X|=|Y|=i-j \\ |X \cap [s]| = p, |Y \cap [r]| = q}} 
\minor{A}{X}{Y}^2 
\]
and so the proposition follows from the change of variable $j \leftarrow i - j$.
\end{proof}
The key observation in \cite{MCohen} is that, while each term $A_{p, q}^j$ is 
itself a sum with an exponential number of terms, the total number of the 
$A_{p, q}^j$ is polynomial in $n$.
Hence if there was a way to compute these terms more efficiently (than the 
brute-force sum used in their definition), then the full polynomial could then 
be computed efficiently as well.
To accomplish this, \cite{MCohen} introduces a multivariate 
characteristic polynomial which will be used as a generating function for these 
terms.
\begin{defn}
For $m, n, s, r$ as in Proposition~\ref{prop:decompose} and a matrix $A \in 
\mathbb{R}^{\mm \times \nn}$, we define the matrix $\tilde{A}(y, z)$ whose 
entries are polynomials in $y$ and $z$:
\[
\tA(y,z)= 
\diag{y I_s}{I_{\mm - s}} A 
\diag{z I_r}{I_{\nn-r}}
\]
and let the polynomial
\begin{equation}\label{eq:def}
\theta_{A} (x,y,z)= \mydet{xI + \tA(y,z)^T \tA(y,z)}
\end{equation}
\end{defn}
The idea will be for $\theta_A(x, y, z)$ to act as a generating function, with 
the $y$ and $z$ variables keeping track of terms involving the upper left 
corner of $A$ while simultaneously allowing us to compute the polynomial as a 
simple determinant (which can be computed efficiently, despite being 
expressable as a sum with an exponential number of terms). 
More precisely we have the following result:
\begin{lemma}\label{lem:same}
\[
\theta_{A} (x,y,z)= \sum_{j=0}^\nn \sum_{p=0}^s  \sum_{q=0}^r x^{\nn-j} 
y^{2p}z^{2q} A_{p,q}^j
\]
where $A_{p,q}^j$ is the same as in Proposition~\ref{prop:decompose}.
\end{lemma}
\begin{proof}
By Corollary~\ref{coeff}, we have
\begin{align*}
\mydet{xI + \tA(y,z)^T \tA(y,z)}
&= \sum_{i=0}^\nn x^{\nn-i} \left(  \sum_{V \in 
 \binom{[\nn]}{i},U\in \binom{[\mm]}{i}}  \minor{\tA(y,z)}{U}{V}^2\right)\\
 &= \sum_{i=0}^d x^{\nn-i}  \sum_{p=0}^s \sum_{q=0}^r \left(  
 \sum_{\substack{|U|=|V|=i \\ |U\cap[s]|=p, |V \cap[r]|=q}}  
 \minor{\tA(y,z)}{U}{V}^2\right).
\end{align*}
where the last line is simply conditioning on the sizes of $|U \cap [s]|$ and 
$|V \cap [t]|$ (and summing over all possibilities).
Since each element of $\tA$ contains a factor of $y$ if and only if it is in 
the first $s$ rows, any $U$ with $|U \cap [s]| = p$ will result in a factor of 
$y^p$.
Similarly each element of $\tA$ contains a factor of $z$ if and only if it is 
in the first $r$ columns, so any set $V$ with $|V \cap [t]| = q$ will result in 
a factor of $z^q$. 
Hence we have
\[
\theta_A(x, y, z) = 
\sum_{k=0}^d x^{\nn-i}  \sum_{p=0}^s \sum_{q=0}^r \left( 
 \sum_{\substack{|U|=|V|=i \\ |U\cap[s]|=p, |V \cap[r]|=q}}  
  y^{2p}z^{2q} \minor{A}{U}{V}^2\right)
\]
and so the result follows by changing the order of summations.
\end{proof}

Hence the $A^j_{p,q}$ terms can be computed in polynomial time using 
\ref{eq:def} as a generating function, thus giving a polynomial time algorithm 
for computing \eqref{eq:expect} as needed.

\begin{rmk}
Recall from Section~\ref{sec:alg} that in the special case $s=\kk\nn$ and 
$r=\nn$ (corresponding to a graph that is a union of $k$-claw matchings), we 
noted that it was possible to compute the expected characteristic polynomial 
efficiently from the (classic) characteristic polynomial of the graph.
In this case, it is easy to calculate $\tA(y,z)= yz A$ so \eqref{eq:def} 
reduces to
\[
\theta_A 
(x,y,z)= 
(yz)^{s}\charpoly{A^TA}(x/{yz})
\]
which is exactly the information we were using before.
Hence the ability to use $\theta_A(x, y, z)$ to compute 
\eqref{eq:alg_goal} seems to be the appropriate generalization of this 
observation.
\end{rmk}

\section{Appendix: Tools from linear algebra}

We will need two formulas that give an expansion of the determinant as a 
function of the minors of matrices.  
The first, which can be found in \cite{sums}, is for the sum of matrices:
\begin{lemma} \label{sumdet}
For all matrices $A, B \in \R^{n\times n}$, one has
\[
\mydet{A+B}= \sum_{\substack{S,T \subseteq [n] \\ |S| = |T|}} 
(-1)^{\|S+T\|_{1}} \minor{A}{S}{T} 
\minor{B}{[n]\setminus S}{[n] \setminus T}
\]
where $S + T$ denotes the symmetric difference (XOR) of $S$ and $T$, and we use 
the notation
$\|X\|_1 
= 
\sum_{x \in X} x$. 
\end{lemma}

The second, which is a well-known result of Cauchy and (independently) Binet, 
regards the product of matrices (see \cite{horn}):
\begin{lemma} \label{proddet} \label{cb}
For all  matrices $A \in \R^{m \times n}$ and $B \in \R^{n \times p}$ and for 
all sets $|S| = |T| = k$ with $k \leq \min \{ m, n, p\}$, we have
\[
\minor{AB}{S}{T}
= \sum_{U \in \binom{[\nn]}{k}} \minor{A}{S}{U} \minor{B}{U}{T}
\]
\end{lemma}

We also need a well-known expansion of the characteristic polynomial in terms 
of the principal minors (see \cite{horn}, but this can also be derived directly 
from Lemma~\ref{sumdet} and Lemma~\ref{cb}):

\begin{cor} \label{coeff}
For a matrix $A \in \R^{n\times n}$, one has
\[
\mydet{xI+A}= \sum_{k=0}^{\nn} x^{\nn-k} \sum_{S \in 
\binom{[\nn]}{k}} 
\minor{A}{S}{S}
\]
\end{cor}

Finally, we use the fact (which we prove here) that random elements of the 
Stiefel manifold $\rframe{r}{s}$ 
under the invariant  (Haar/uniform) measure is {\em minor-orthogonal}. 
\begin{lemma} \label{frameLA}
Let $Q$ be a random matrix drawn uniformly from $\rframe{r}{s}$.
Then for any sets $S, T, U, V$ with $|S|=|T| = i$ and $|U|=|V|=\ell$, we 
have
\[
\expect{Q}{ \minor{Q}{S}{T}\minor{Q}{U}{V}} = \frac{1}{\binom{s}{i}} 
\test{S=U}\test{T=V}
\]
\end{lemma}
\begin{proof}
Define the matrix $E_t$ to be the diagonal matrix with diagonal entries $1$ 
except for $E_t(t, t)$ which is $-1$.
Now assume that $S \neq U$ and let $h \in S \setminus U$.
Then 
\[
 \minor{E_t Q}{S}{T} \minor{E_t Q}{U, V} = -\minor{Q}{S}{T}\minor{Q}{U}{V}
\]
and so 
\[
\expect{Q}{\minor{E_t Q}{S}{T} \minor{E_t Q}{U}{V}} = - \expect{Q}{ 
\minor{Q}{S}{T}\minor{Q}{U}{V}}
\]
however the function 
\[
f(X) = \minor{X}{S}{T} \minor{X}{U}{V}
\]
is continuous and so by the left invariance of $\rframe{r}{s}$ (see 
Remark~\ref{rmk:stiefel}), we have
\[
\expect{Q}{\minor{E_t Q}{S}{T} \minor{E_t Q}{U}{V}} = \expect{Q}{ 
\minor{Q}{S}{T}\minor{Q}{U}{V}}.
\]
Hence we must have 
\[
\expect{Q}{ \minor{Q}{S}{T}\minor{Q}{U}{V}} = 0
\]
and similarly when $T \neq V$.

In the case that $S = U$ and $T = V$, we must compute
\[
\expect{Q}{ \minor{Q}{S}{T}^2}.
\]
Again, by the invariance properties of $\rframe{r}{s}$ (this time under 
multiplication by permutation matrices on each side), this quantity is 
independent of the actual $S$ and $T$ and instead depends only on $|S|=|T| = j$.
On the other hand, by Lemma~\ref{cb}, we have for any $S$ with $|S| = j$
\[
1 = \minor{Q^T Q}{S}{S} = \sum_{U \in \binom{[s]}{j}} \minor{Q}{S}{U}^2
\]
since $Q^TQ = I_r$ (the $r \times r$ identity matrix).
Since there are $\binom{s}{j}$ terms in this sum (and all are equal in 
expectation), we conclude that
\[
\expect{Q}{ \minor{Q}{S}{T}^2 } = \frac{1}{\binom{s}{j}}.
\]
\end{proof}

\end{document}